\documentclass[11pt]{article}

\usepackage{amsthm}
\usepackage{amsfonts}
\usepackage{amsmath}
\usepackage{amssymb}
\usepackage{enumerate}
\usepackage{graphicx}
\usepackage{subcaption}
\usepackage{mathtools}

\usepackage{latexsym,comment,setspace,array,varwidth,color,float}

\usepackage{blkarray}
\usepackage{multirow,color}

\newtheorem{theorem}{Theorem}
\newtheorem{lemma}[theorem]{Lemma}
\newtheorem{corollary}[theorem]{Corollary}
\newtheorem{proposition}[theorem]{Proposition}

\theoremstyle{definition}
\newtheorem{definition}[theorem]{Definition}

\theoremstyle{remark}

\usepackage[sc]{mathpazo}
\usepackage[T1]{fontenc}
\usepackage[utf8]{inputenc}

\usepackage[a4paper,bindingoffset=0.2in,%
left=1in,right=1in,top=1in,bottom=1in,%
footskip=.25in]{geometry}

\begin{document}

\title{\vspace{-2cm} {Nullspace Vertex Partition in Graphs} 
\author{Irene Sciriha
\thanks{{irene.sciriha-aquilina@um.edu.mt}
{:--Corresponding author}}
\and{Xandru Mifsud\thanks{xandru.mifsud.16@um.edu.mt}}\and
{James Borg\thanks{james.borg@um.edu.mt}}
 \thanks{\textsc{Department of Mathematics, 
 Faculty of Science, University of Malta, Msida, Malta}}}}

\maketitle
{\bf Keywords} {Nullspace, core vertices, core--labelling, graph perturbations.}

{\bf Mathematics Classification 2021}{05C50, 15A18}


\begin{abstract}
	 The core vertex set of a graph is an invariant of the graph. It consists of those vertices associated with	  the non-zero entries of the nullspace vectors of a  $\{0,1\}$-adjacency matrix.  The remaining vertices of the graph form  the core--forbidden vertex set. For  graphs with independent core vertices, such as bipartite minimal configurations and trees, the nullspace induces a well defined three part vertex partition. The parts of this partition are  the core vertex set, their neighbours and  the remote core--forbidden vertices. The set of the remote core--forbidden vertices are those not adjacent to any core vertex. We show that this set can be removed, leaving the nullity unchanged. For graphs with independent core vertices,  we show that  the submatrix of the adjacency matrix defining the edges incident to the core vertices determines the nullity of the  adjacency matrix.  For the efficient allocation of   edges  in a  network graph without altering the  nullity of its adjacency matrix, we determine  which perturbations make up  sufficient conditions for the  core vertex set of the adjacency matrix of a graph to be preserved in the process.

\end{abstract}

\section{Introduction}
\label{SecIntro}
A graph $G = (V,E)$ has  a finite vertex set $V=\{v_1,v_2,\ldots,v_n\}$ with vertex labelling $[n]:=\{1,2,...,n\}$ and an edge set $E$ of 2-element subsets of $V$. 
 The graphs we consider are simple, that is without loops or multiple edges. A subset $U$ of $V$ is independent if no two vertices form an edge. 
 The open--neighbourhood of a vertex $v \in V$, denoted by $N(v)$, is the set of all vertices incident to $v$.
   The degree $\rho (v)$ of a vertex $v$ is the number of edges incident to $v$. 
 The induced subgraph $G[V\backslash S]$ of $G$   is $G-S$ obtained 
 by deleting a vertex subset $S,$  together with  the edges incident to the vertices in $S$. For simplicity of notation,  we write $G-u$ for the induced subgraph obtained from $G$ by deleting vertex $u$ and $G-u-w$ when  both vertices $u$ and $w$ are deleted.

The adjacency matrix $\mathbf{A} = (a_{ij})$ of the labelled graph $G$ on $n$ vertices is the $n\times n$ matrix $\mathbf{A} = (a_{ij})$  such that $a_{ij} = 1$ if the vertices $v_i$ and $v_j$ are adjacent (that is $v_i\sim v_j$)  and $a_{ij} = 0$ otherwise.  The nullity $\eta(G)$ is the algebraic 
multiplicity of the eigenvalue $0$ of $\mathbf{A}$, obtained as a root of the characteristic polynomial $\det(\lambda {\bf I} -{\bf A})$.  The geometric multiplicity of an eigenvalue of a matrix is the dimension   $\eta({\bf A})$ of the nullspace $\ker({\bf A})$  of ${\bf A}.$     
 Since ${\bf A}$  is real and symmetric,  it is the same as  its  algebraic 
multiplicity. In particular, the {\it nullity }  $\eta(G)$  of $G$ is the multiplicity of the eigenvalue 0.
By the dimension theorem for  linear transformations, for   a graph  $G$ on $n$ vertices, the  rank of  ${\bf A}$ is $\text{rank}(G) = n - \eta(G)$.
 Graphs, for which 0 is an eigenvalue, that is  $\eta(G)> 0,$  are singular.

   In \cite{ SciCoefx97, SciConstrNullOne98,SciRanKalamazoo1999}, the terms core vertex, core--forbidden vertex and kernel vector for a singular graph $G$  are introduced. The \textit{kernel vector} refers to a non--zero vector $\mathbf{x}$ in the nullspace of $\mathbf{A},$ that is, it satisfies $\mathbf{A}\mathbf{x} = \mathbf{0}$, $\mathbf{x} \neq \mathbf{0}$. The support of a vector  $\mathbf{x}$  is the set of indices of non--zero entries of ${\bf x}$. 

\begin{definition}
	\label{core}
	\cite{SciConstrNullOne98,SciCHznSingGr07} A vertex of a singular graph $G$ is a {\em core vertex} $(cv)$ of  $G$ if it  corresponds to a non-zero entry of {\it some} kernel vector of $G$. A vertex $u$ is a {\em core--forbidden vertex} $(cfv)$, if {\it every} kernel vector has a zero entry at position $u.$
\end{definition}

It follows that the union of the elements of the support of  
all kernel vectors of  ${\bf A}$ form the set of core vertices of  $G.$ It is 
clear that
 a vertex of a singular graph $G$ is either a $cv$ or a $cfv$. The set of core vertices is denoted by $CV$, and the set  ${\mathcal V}\backslash CV$ by $CFV$. 
 
 Cauchy's Inequalities for real symmetric matrices, also referred to as the   Interlacing Theorem in spectral graph theory \cite{SelectedSchwenkEigs}, are considered to be among the   most powerful tool in studies related to the location of eigenvalues. The   Interlacing Theorem     refers to the interlacing of the eigenvalues of the adjacency matrix of a vertex  deleted subgraph relative to those of the parent graph.

As a consequence of  the well-known  Interlacing Theorem, the nullity of a graph can change by 1 at most,  on deleting a vertex.

On deleting a vertex, the nullity  reduces  by 1 if and only if the vertex is a core vertex \cite[Proposition~1.4]{SciMaxExtremSing2012}, \cite[Corollary 13]{SciCoalEmbNut08} and \cite[Theorem 2.3]{SciFarrBk2020}. It follows that the deletion of a core--forbidden vertex can leave the nullity of the adjacency matrix unchanged, or else  the nullity increases by 1.

\begin{definition}  A vertex of a graph $G$ is  $cfv_{mid}$ if its deletion leaves the nullity of the adjacency matrix, of the subgraph obtained, unchanged. A vertex of $G$ is 
	 $cfv_{upp}$ if when  removed,  the nullity increases by 1.
The set  $CFV$ is the disjoint union of the sets $\left\{cfv_{mid}\right\}$ and $\left\{cfv_{upp}\right\}$, denoted by $CFV_{mid}$  and $CFV_{upp},$ respectively.
\end{definition}

At this point it is worth mentioning that in 1994, the first author coined the phrases {\it core vertices,  periphery and  core-forbidden vertices}. 
The core vertices  with respect to  ${\bf x}$ of a graph $G$ with a singular adjacency matrix ${\bf A}$ correspond to the support 
of the vector ${\bf x}$ in the nullspace of  ${\bf A}$. 
One must not confuse the  core vertex set with the same term referring to independent sets introduced much  later  \cite{Levit2002Core}. 
The term core is also used in relation to graph homomorphisms. 

There were other researchers who used associated concepts in different contexts. 
In 1982, Neumaier  used the terms  {\it essential} and {\it non--essential} vertices 
corresponding to core vertices and core--forbidden vertices, respectively, but  only for the class of trees \cite{NeumaierEssential82}.
Back in 1960, S. Parter studied the upper core--forbidden vertices in the context of real symmetric matrices. 
In fact in the linear algebraic community these vertices are referred to as Parter vertices, the core vertices 
 as downer vertices and the middle core--forbidden vertices as neutral vertices. Core--forbidden vertices 
are also referred to as Fiedler vertices in engineering.

Graphs with no edges between pairs of vertices in CV have a well defined vertex partition, which facilitates the form of the adjacency matrix in block form as  shown in (\ref{EqCoreL}) in Section \ref{SecInd}.

 \begin{figure}
 	\begin{center}
 		\includegraphics[width=7cm]{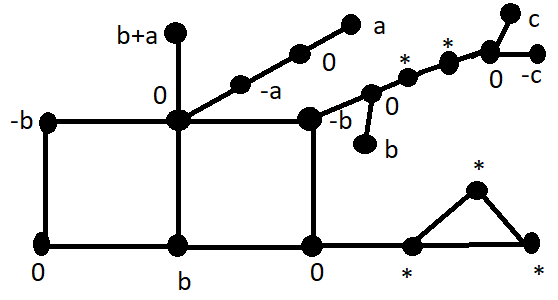}
 		\caption{A vertex partition induced by a generalized kernel vector of $G$ in a graph of nullity 3: the label 0 indicates a vertex in $N(CV)$; the starred vertices are $cfv_R$.}   \label{FigNull2e}
 	\end{center}
 \end{figure}
\begin{definition}
A graph is said to have \textit{independent core vertices} if no two core vertices are adjacent.
\end{definition}

If  CV is  an independent set, then 
 the  core-forbidden vertex set $CFV$ is partitioned into two subsets:
 $N(CV),$ the neighbours of the core vertices in $G$, and  $CFV_{R},$  the remote core--forbidden  vertices, as shown in Figure \ref{FigNull2e} for a graph of nullity 3.
  A similar concept is considered in \cite{SANDER2009133,JAUME2018836} for the case of trees. In this work, unless specifically stated,  we consider all graphs.
 
 \begin{definition}
 A \textit{core-labelled} graph $G$ has   an independent CV. The  vertex set of $G$ is partitioned such that $V = CV \ \dot{\cup} \ N(CV) \ \dot{\cup} \ CFV_R$. The vertices of $CV$ are labelled first, followed by those of $N(CV)$ and then by those of $CFV_R.$\end{definition}

In Section \ref{SecPend}, we  show that removing a pendant edge from a graph not only preserves the nullity (which is well known)  but also the type of vertices. 

 In Section \ref{SecInd}, 
we determine the nullity of the submatrices of the adjacency matrix for a graph in  the class of graphs with independent core vertices.  The remote core--forbidden vertices do not contribute to the equations involving the  nullspace vectors and can be removed to obtain a {\it slim graph}. In Section \ref{SecBip}, bipartite minimal configurations are shown to be slim graphs with independent core vertices. Moreover all vertices in $N(CV) $ of a bipartite minimal configuration are shown to be  upper core--forbidden vertices. 
In Section \ref{SecTrees}, we obtain results on the nullity and the number of the different types of vertices of singular trees in the light of the results obtained in  Section \ref{SecInd}.  Section \ref{SecAddE} focuses on  the types of non--adjacent vertex pairs that can be joined by edges in  a graph under various constraints associated with the nullspace of ${\bf A}.$

\section{Graphs with Pendant Edges} \label{SecPend}
By definition of $CV$  and $CFV,$ the nullspace of $\bf A$ induces a partition of the vertices of the associated graph $G$ into $CV$ and $CFV.$  The set $CV$ is empty if $G$ is non--singular  and non--empty otherwise. It could happen that  $CFV$ is empty  in which case the  graph is  singular and it is a {\it core graph}. Consider two graphs on 4 vertices. The path $P_4$ is non--singular whereas the cycle $C_4$ is a core graph of nullity two.

A quick method  to obtain the nullity and kernel vectors of a graph is known as the {\it Zero Sum Rule}.  The neighbours of a vertex  are weighted so that their weights add up to zero. Repeating this for each vertex gives the minimum number $\eta (G)$ of independent parameters in which to express the entries of a generalized vector in the nullspace of ${\bf A}$. Figure \ref{FigNull2b} shows a graph of nullity two with the entry of a generalized kernel vector next to each vertex, in terms of the parameters $a$ and $b$.  
   
\begin{figure}[!h]
\begin{center}
\includegraphics[width=7cm]{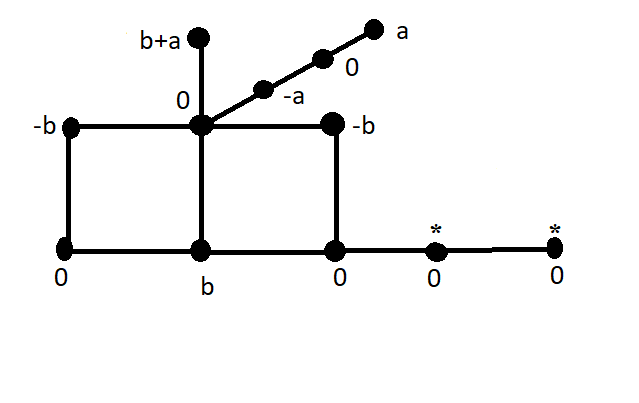}
\caption{A graph of nullity 2 and  a generalized kernel vector ${\bf x}(a,b)$ of $G$. The labels in terms of $a$ and $b$ identify the core vertices;  the label 0 indicates  core--forbidden  vertices.}  \label{FigNull2b}
\end{center}

\end{figure}

We are interested in the change in the type of vertices on the deletion  of   vertices and edges. 
Deleting a core--vertex from an odd path $P_{2k+1}$ may transform some of  the core vertices to $CFV_{upp}$. Similarly, deleting a $CFV_{upp}$ vertex from the cycle $C_6$  on six vertices transforms some of the core--forbidden vertices to core--vertices. 
Removing a core vertex and a neighbouring $cfv$ may alter the nullity. Consider the 4 vertex graph obtained by  identifying an edge of two 3--cycles. Removing the identified edge increases the nullity by 1, whereas removing any of the other edges decreases the nullity by 1.

 However, it is well known that removing an end vertex $v,$  also known as a leaf, in the literature, and its unique neighbour $u,$ from a graph $G$, leaves the nullity unchanged in $G - u - v$ \cite{CvetGutMultZeroeig72}. Note that  the vertex $v$ may be  $cv$  or $cfv.$ We give a new proof of this known result that also leads to an unusual preservation of the  type of the remaining vertices after removing two vertices.

\begin{theorem} \label{TheoRemPend}
Let $w$ be an end vertex and $u$ its unique neighbour in a singular graph $G$. The nullity of $G - u - w$ is the same as that of $G$. Moreover, the type of vertices in $G - u - w$ is preserved. 
\end{theorem}

\begin{proof}
Let $u, w$ be the $n-1^{\text{th}}$ and  $n^{\text{th}}$ labelled vertices, respectively, of a graph on $n$ vertices.  The adjacency matrix $\mathbf{A}(G)$ satisfies  $$
\mathbf{A}(G)\left(\begin{array}{c}
\mathbf{x}\\
\hline
y\\
z \end{array}\right)
=
\left (\begin{array}
{ccc|cc}
&\mathbf{A}(G-u-w)&&{\star}&{\bf 0}\\
  \hline
&{\star}&&0&1\\
&{\bf 0}&&1&0
\end{array}\right )
\left(\begin{array}{c}
\mathbf{x}\\
\hline
y\\
z \end{array}\right)=
\left(\begin{array}{c}
\mathbf{0}\\
\hline
0\\
0 \end{array}\right).$$

Hence $y$ is 0 and $\mathbf{A}(G-u-w)\mathbf{x}=\mathbf{0}.$  Also $z$ depends on $\mathbf{x}$  and the neighbours of $w.$  
The nullity of
$G-u-w$ 
is equal to the nullity of $G$. This is because there is a 1--1 correspondence
between the kernel vectors in $G-u-w$ and the kernel vectors in $G$. Whatever $z$ is, this 1--1 correspondence holds. So the number of linearly independent
vectors in the nullspace of G is equal to the number of linearly independent vectors 
in the nullspace of  $G-u-w.$
 Also, on removing the end vertex and its neighbour,
the non--zero entries of $\mathbf{x}$ restricted to  $G-u-w$ 
will be the same as for  $G$. Hence, the core  and core--forbidden vertices in $G -u-w$
are the same as those in $G$.
\end{proof} 

In a tree, it is possible to remove  end vertices and associated unique neighbours successively until no edges remain. Indeed, 
the  graph obtained by removing all pendant edges in $T$ and in the subgraphs obtained in the process, is $\overline{K_{\eta}}, $ each vertex of which, as expected from Theorem \ref{TheoRemPend}, is a core vertex.
This leads to a well known criterion to determine the nullity of a tree.

\begin{corollary} \label{CorTreeRem2v}  For a tree $T,$  the number of isolated vertices, obtained by the removal of  end vertices and their unique neighbours in  $T$ and in its successive subgraphs, is
 $\eta(T).$
\end{corollary}


 Since by Theorem \ref{TheoRemPend}, the vertices of $\overline{K_{\eta}}$ are in CV of $T,$ we can deduce the following result:

\begin{proposition}
	\label{PropT2endCV}
	A singular tree $T$ has at least 2 core vertices  which are end vertices.	
\end{proposition}

\begin{proof}
Starting from any end--vertex in $T,$  if the order of  pendant--edge removals, 
is chosen appropriately, then at least one vertex $u$ of  $\overline{K_{\eta}},$  obtained as in Corollary \ref{CorTreeRem2v}, is  an end--vertex of $T$  and its type in  $T$   is a {\it cv}. 

Similarly, starting from the edge containing the end--vertex $u$ of $T,$   there is  another end--vertex $w$ which is a {\it cv}  of $T.$

%

	\end{proof}

Corollary \ref{CorTreeRem2v} describes  a polynomial--time algorithm to determine the nullity of a tree. A {\it matching} in a bipartite graph is a set of  edges, no two of which share a common vertex. The matching number $t$ is the number of edges in a maximal matching \cite{CvetGutMultZeroeig72}.  Corollary \ref{CorTreeRem2v} and Proposition \ref{PropT2endCV} provide an  immediate  proof of the well known result $\eta(T)=n-2t$ \cite{CvetGutMultZeroeig72}.

\section{Graphs with independent core vertices} \label{SecInd}

In a singular graph, core vertices may be adjacent. Indeed, in a core graph (not $\overline{K_r}$), each edge joins two core vertices. The family of cycles $C_{4k}, k\in {\mathbb N}$ consists of core graphs of nullity 2.

By definition, a singular graph has a non--empty $CV.$ If in a singular graph,  $N(CV)$  is empty, then $CFV$ must be empty and the graph is a core graph. 

 It is convenient to work with graphs for which $CFV_R$  is empty. Removal of  $CFV_R$ from a graph leaves the type of vertices in the resulting subgraph unchanged. 
\begin{definition}\label{DefSlim} A connected singular graph $G$ is a {\it slim graph} if it has an independent $CV$ and $CFV$  is precisely $N(CV)$.
\end{definition}
From Definition \ref{DefSlim}, 
it follows that  a  singular graph is slim if and only if its $CV$ is an independent set and its  $CFV_R$  is empty.

For a core--labelled graph
  the adjacency matrix $\mathbf{A}$  is a block matrix of the form,
\begin{equation}
\mathbf{A} = \left[\begin{array}{c|c|c}
\mathbf{0} & \mathbf{Q} & \mathbf{0} \\ \hline
\mathbf{Q}^\intercal & \mathbf{N} & \mathbf{R} \\ \hline
\mathbf{0} & \mathbf{R}^\intercal & \mathbf{M}
\end{array}\right] \label{EqCoreL}
\end{equation}
where $\mathbf{Q}$ is $CV \times N(CV)$, $\mathbf{R}$ is $N(CV) \times CFV_R$, $\mathbf{N}$ is $N(CV) \times N(CV)$ and $\mathbf{M}$ is $CFV_R \times CFV_R$. The submatrix $\mathbf{Q}$ plays an important role to relate the linear independence of its columns to the nullity of $G$.

\begin{lemma} \label{LemNullQ}
	Let $G$ be a singular core--labelled  graph. Then
	$\eta(\mathbf{Q}^\intercal)=\eta(G)$.
	
\end{lemma}

\begin{proof}
	For a core--labelling of $G$, let $\mathbf{x}^{(i)}$ be one of the $\eta(G)$  kernel vectors of $\mathbf{A}$. The vector $\mathbf{x}^{(i)}$  is of the form $\left(\mathbf{x}_{CV}^{(i)}, {\bf 0}\right)$  and $\mathbf{    x}_{CV}^{(i)} = \left(\alpha_1,...,\alpha_{|CV|}\right) \not ={\bf 0}.$   Now, $\mathbf{Ax}^{(i)} = \mathbf{0}$ if and only if $\mathbf{Q}^\intercal \mathbf{x}_{CV}^{(i)} = \mathbf{0}$. Thus there are as many linearly independent kernel vectors of ${\bf A}$ as there are of $\mathbf{Q}^{\intercal}.$  It follows that $\textup{Dim}\left(\textup{Ker}(\mathbf{Q}^\intercal)\right) = \textup{Dim}\left(\textup{Ker}(\mathbf{A})\right)$.
\end{proof}

\begin{lemma} \label{LemQdep} Let $G$ be  a singular core--labelled  graph.  
	For a core--labelling of $G$, the columns of $\mathbf{Q}^{\intercal}$ are linearly dependent and $\textup{rank}(\mathbf{Q}) < |CV|$.
\end{lemma}

\begin{proof}
	Since $\text{Dim}\left(\text{Ker}(\mathbf{A})\right) \geq 1$, then $\text{Ker}(\mathbf{Q}^\intercal)\not =\{{\bf 0}\}.$   Thus there is a non-zero linear combination of the columns of $\mathbf{Q}^\intercal$ that is equal to $\mathbf{0}$, that is  $\mathbf{Q}^\intercal \mathbf{x}_{CV} = \mathbf{0}$. Hence  the columns of $\mathbf{Q}^\intercal$ are  linearly dependent. Since column rank is equal to row rank, it follows that $\textup{rank}(\mathbf{Q}) < |CV|$.
\end{proof}

\begin{figure}[htbp!]
\begin{center}
\includegraphics[width=4cm]{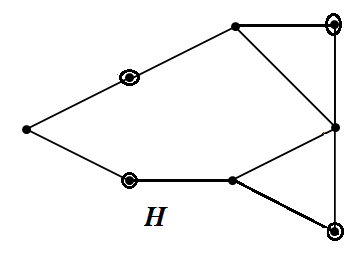} \hspace*{1cm}
\includegraphics[width=4cm]{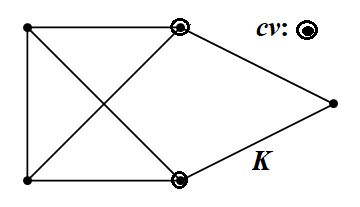}
\caption{In graph $H,$\,the number of vertices in CV and in NCV are the same and in graph $K,$\,$|CV|<|N(CV)|$.}  \label{FigNCV}
\end{center}
\end{figure}

The relative number  of vertices in $CV$ and in $N(CV)$ may differ. 
For  the graphs $H$ and $K$ of  Figure \ref{FigNCV}  $|CV|=|N(CV)|$ and   $|CV|<|N(CV)|, $  respectively.  In Section \ref{SecBip},  we see that 
graphs with 
   $|CV|>|N(CV)| $  exist, a property satisfied by  minimal configurations (defined in Definition \ref{Defmc}).

\begin{theorem} \label{TheoNullCVrnkQ}
	Let $G$ be a  singular core-labelled graph with independent core vertices. Then $\eta(G) = |CV| - \textup{rank}(\mathbf{Q})$.
\end{theorem}

\begin{proof}  By the well known dimension theorem,\\ \hspace*{2cm} Dim(Domain$({\bf Q}{^\intercal}))=$ Dim(Ker$({\bf Q}{^\intercal}) ) +$ Dim(Im$({\bf Q}{^\intercal})).$

Now Dim(Domain$({\bf Q}{^\intercal}))= |CV|.$  By Lemma \ref{LemNullQ},  Dim(Ker$({\bf Q}{^\intercal}) )=\eta(G).$ Hence rank$({\bf Q})= $  rank$({\bf Q}^{\intercal})= |CV|   - \eta (G).$  \end{proof}

It is clear that for a singular core--labelled graph, if  $|CV|< |N(CV)|,$  then the columns of the  $|CV|\times  |N(CV)|$ matrix  ${\bf Q}$  are linearly dependent. For $|CV|= |N(CV)|,$ by Theorem \ref{TheoNullCVrnkQ},    $\textup{rank}(\mathbf{Q}) < |CV| $  and thus  the $|N(CV)|$ columns of $\mathbf{Q}$  are linearly dependent.  
We shall now determine a necessary and sufficient condition for $\mathbf{Q}$  to have full column rank. 

\begin{theorem} \label{TheoQindep} Let $G$ be  a singular core--labelled graph. 
	The matrix $\mathbf{Q}$ has linearly independent columns if and only if $\eta(G) = |CV| - |N(CV)|$.
\end{theorem}

\begin{proof}  The matrix $\mathbf{Q}$  has full rank  if and only if rank$({\bf Q})= {\rm Dim}({\rm Im}({\bf Q}))=|N(CV)|.$   By Theorem \ref{TheoNullCVrnkQ}, the necessary and sufficient condition for the matrix $\mathbf{Q}$ to have linearly independent columns  is that $\eta(G) = |CV| - |N(CV)|$.  \end{proof}

Recall that the vertex set $V$ of a core--labelled graph is partitioned into 
 $CV$, $N(CV)$  and $CFV_R.$ On  deleting $N(CV)$ and $CV$ from a graph, the subgraph induced by $CFV_R$  remains. 

\begin{theorem}\label{TheoCFVrInv}
The subgraph induced by $CFV_R$  for a core--labelled graph  is non--singular.
\end{theorem}

\begin{proof}
Using an  adjacency matrix ${\bf A}$ of the form (\ref{EqCoreL}), we need to show that $\mathbf {My=0}$ if and only if $\mathbf{y=0}.$ For a core--labelling, all kernel vectors of $\textbf{A}(G)$ are of the form ${\bf z}=\left(\begin{array}{c}
{\bf x}\\
{\bf 0}\\
{\bf 0}
\end{array}\right).$

But  ${\bf My=0}$ for some ${\bf y}\not ={\bf 0} $  if and only if  there exists   ${\bf x} $ such that ${\bf A}\left(\begin{array}{c}
{\bf x}\\
{\bf 0}\\
{\bf y}
\end{array}\right)={\bf 0}.$ This contradicts the form of the kernel vector for a core--labelling. Hence no kernel vectors exist for ${\bf M}.$
\end{proof}

Graphs with independent core vertices include the family of half cores. A \textit{half core} is a  bipartite graph with one partite set being the set $CV$  and the other partite set being $CFV$. In Section \ref{SecTrees}, we shall see that trees also have independent core vertices.

At this stage, the case for unicyclic graphs is worth mentioning.  The coalescence of two graphs is obtained by identifying a vertex of one graph with a vertex of the other graph. If none of the two graphs is $K_1$, then this vertex becomes a cut vertex.  Unicyclic graphs can be considered  
  to be the coalescence of a cycle $C_r$   with $r$ trees (some or all of which  may be the isolated vertex  $P_1$), each tree $T_v$ coalesced with $C_r$ at a unique vertex $v$ of the cycle.  If $r\not =4k,\ k\in {\mathbb Z}^+, $ then the unicyclic graph has independent core vertices. Since the nullity of $C_4$ is 2, using Theorem \ref{TheoRemPend}, the following result is immediate.
  
  \begin{theorem}
 \label{TheoUnicyc}  Let $G$ be a unicyclic graph  with cycle $C_r$ where $r=4k.$ 
\begin{enumerate}[{\rm (i)}]
	\item   
 If  the vertex $v$ of at least one  tree $T_v$ which is coalesced with the cycle is  a core--forbidden vertex, then the unicyclic graph also has  independent core vertices. 
 \item  If  the vertices $v$ of each   tree $T_v$ which is coalesced with the cycle is  a core vertex, then the unicyclic graph must have nullity at least 2. 
\end{enumerate}
 \end{theorem}

\section{Bipartite Minimal Configurations}  \label{SecBip}
In \cite{SciCoefx97, SciConstrNullOne98,SciCHznSingGr07}, the concept of {\it minimal configurations} (MCs)  as  admissible subgraphs, that go to construct a singular graph, is introduced. It is shown that there are $\eta $  MCs as subgraphs of a singular  graph $G$ of nullity $\eta >0.$  
A MC is a graph of nullity 1 and its adjacency matrix ${\bf A} $  satisfies ${\bf Ax}={\bf 0}$ where ${\bf x}\not ={\bf 0}$  is the  generator  of the nullspace of  the adjacency matrix ${\bf A}$  of $G$. The core vertices of a MC induce a subgraph termed the {\it core} $F$  with  respect to ${\bf x}.$ Among singular graphs  with core $F$ and kernel vector ${\bf x},$ a MC has the least number of vertices and  there are no edges joining pairs of core--forbidden vertices. For instance,  the path $P_7$  on 7 vertices is a MC with ${\bf x}=(1,0,-1,0,1,0,-1)^{\intercal}.$

\begin{definition}
 \label{Defmc}
	A \textit{minimal configuration} (MC) is a singular graph on a vertex set $V$ which is either $K_1$ or if $|V| \geq 3$, then it has a core $F = G\left[CV\right]$ and periphery $\mathcal{P} =V\backslash CV$ satisfying the following conditions, \vspace{-3mm}
	\begin{enumerate}[\rm (i)]
		\item $\eta(G) = 1$,
		\item $\mathcal{P} = \emptyset$ or $\mathcal{P}$ induces a graph consisting of isolated vertices,
		\item $\left|\mathcal{P}\right| + 1 = \eta\left(F\right)$.
	\end{enumerate}
\end{definition}

Note that a MC $\Gamma $ is connected. To see this, suppose  $\Gamma $  is the disjoint union $G_1\dot \cup G_2$ of the graphs  $G_1$  and $ G_2,$ labelled so that the core vertices of $G_1$  are  labelled first followed by its {\it cfv}, then the cv of $G_2$  followed by its \textit{cfv}. There exists  a nullspace vector $({\bf x}_1,{\bf 0},{\bf x}_2,{\bf 0}),$ of ${\bf A}$ with each entry of ${\bf x}_1$  and of ${\bf x}_2$ non-zero. Since $({\bf x}_1,{\bf 0},{\bf 0},{\bf 0}),$ and $({\bf 0},{\bf 0},{\bf x}_2,{\bf 0}),$  are conformal linearly independent vectors in the nullspace of 
 ${\bf A},$  the nullity of $G$ is at least 2, a contradiction.  For the nullity to be 1, it follows without loss of generality, that $\mathbf{x}_2={\bf 0}. $  But then all vertices in $G_2$ lie in the periphery and by definition of MC, they form an independent set. Hence $G_2$ consists of isolated vertices that add $|G_2|\ (>0)$ to the nullity of $G_1$, a contradiction. Hence $G$ must consist  of one component only.  

 The $n$--vertex set of a  bipartite graph $G(V_1,V_2,E)$ is partitioned into independent sets $V_1$   and $V_2$ and has edges in $E$  between vertices in $V_1$ and vertices in $V_2.$ 
	 If the vertices in $V_1$ are labelled first, then the adjacency matrix of $G$ is of the form 
\begin{equation}
\mathbf{A} = \left(\begin{array}{c|c}
\mathbf{0} & \mathbf{S} \\ \hline
\mathbf{S}^\intercal & \mathbf{0}
\end{array}\right),  \label{EqBip}\end{equation} where the $|V_1| \times |V_2|$ matrix $\mathbf{S}$ describes the edges between $V_1$ and $V_2$. The nullity of ${\bf A}$ is  $n-2\  {\rm rank}({\bf S }). $ We have proved the following result:

\begin{proposition}\label{PropBipParity}
The nullity of the adjacency matrix of an $n$--vertex  bipartite graph and $n$ are of the same parity.
\end{proposition}

 In \cite{BevisRank95}, the result in Proposition \ref{PropBipParity} is obtained for  trees, a subclass of the bipartite graphs.  In particular, a bipartite non-singular graph has an even number of vertices.

To explore bipartite MCs it is convenient to consider first a singular bipartite graph of nullity 1. 
\begin{proposition}  \label{PropBipCVInd}
A singular  bipartite graph of nullity 1 admits a core--labelling.
\end {proposition}

\begin{proof}  Let $G({ V}_1, { V}_2, { E})$ be a singular bipartite graph with partite sets ${ V}_1$   and ${ V}_2.$
We show that $CV\subseteq { V}_1,$ without loss of generality.

Suppose   $CV\subseteq { V}_1 \cup  { V}_2.$  Then there exists   ${\bf x}=\left(\alpha_1, ..., \alpha_{|{ V}_1|}, \beta_1, ..., \beta_{|{ V}_2|}\right)^\intercal , $  $\ {\bf x}\not ={\bf 0},$   where not all the $\alpha_i$ are zero and not all the $\beta_j$  are zero.
Then  ${\bf A}\left(\alpha_1, ..., \alpha_{|{ V}_1|}, 0,...,0\right)^\intercal={\bf 0}$  and ${\bf A}\left( 0,...,0,  \beta_1, ..., \beta_{|{ V}_2|}\right)^\intercal={\bf 0},$ showing that  ${\bf A}$  has two linearly independent nullspace vectors. This contradicts that the nullity of a MC is 1.

Hence without loss of generality, $\beta _j=0,\ 1\leq j\leq |{ V}_2|,$  showing that the core vertices lie in ${ V}_1.$ Thus the  $CV$ of a bipartite MC  is necessarily an independent set, which is the condition for the existence of a core--labelling.
\end{proof}

\begin{theorem} \label{TheoBipNull1}
	Let $G$ be a bipartite graph, of nullity 1, on $n$ vertices with partite vertex sets $V_1$ and $V_2$. Then, \vspace{-5mm}
	\begin{enumerate}[\rm (i)]
		\item $n$ is odd
		\item For $|V_1| > |V_2|$, $|V_2| = \dfrac{n-1}{2}$ and $|V_1| = |V_2| + 1$
		\item $ CV \subseteq   V_1.$  
	\end{enumerate}
\end{theorem}

\begin{proof} Let the adjacency matrix   of $G$  be as in (\ref{EqBip}).
	
	\begin{enumerate}[(i)]
		\item Since $\text{rank}(\mathbf{A}) = 2\, \text{rank}(\mathbf{S})$ and $\eta(G) = 1$, then  $n = 2\, \text{rank}(\mathbf{S}) + 1$, which is odd.
		
		\item Without loss of generality, let $|V_1| > |V_2|$. Then $\text{rank}(\mathbf{S})\leq |V_2|$.
Hence 		$n - 1 = \text{rank}(\mathbf{A}) \leq 2|V_2|.$ 
Thus $ |V_1| + |V_2| - 1 \leq 2|V_2|$  and$|V_1| = |V_2| + 1$. Since $n = |V_1| + |V_2|$, it follows that $|V_2| = \dfrac{n - 1}{2}$.

\item 
The proof of Proposition \ref{PropBipCVInd}  shows that $CV\subseteq { V}_1.$ 
\end{enumerate}
\vspace{-8mm}
\end{proof}
A MC has nullity equal to 1. For a  bipartite MC, with partite sets $V_1$  and  $V_2,$ and   $|V_1|>|V_2|,$  we have  
  $|V_1|=|V_2|+1.$ 
\begin{corollary}
Let $G$ be a bipartite MC with vertex partite sets $V_1$  and $V_2,$   where   $|V_1|>|V_2|.$ 
		Then  the set  $CV$  of core vertices is $V_1$ and the  set $CFV$  ( that is $\mathcal{P}$)  is $V_2$.
\end{corollary}

\begin{proof}
By Theorem \ref{TheoBipNull1}(iii),  $CV\subseteq { V}_1.$ 
  A minimal configuration is connected and  $V_1$ is   
   an independent set in a bipartite MC. Note that  $\mathcal{P}$ is an  independent set.  Thus the only neighbouring vertices of a vertex in   $\mathcal{P}$ are in $CV.$
Since  $ \mathcal{P}= V\backslash CV,$ then  $ \mathcal{P}\cap V_1=\emptyset. $ 
Thus  $\mathcal{P}\subseteq V_2$. Moreover $CV=V_1$  and $\mathcal{P}=V_2.$
\end{proof}

Another characterization of a bipartite MC focuses on the removal of extra vertices and edges,  from  a singular bipartite graph of nullity 1, producing a slim graph (Definition \ref{DefSlim}, page \pageref{DefSlim}).

\begin{theorem}  \label{TheoBipMCharzn}
A graph  $G(V_1,V_2,E),$  \   $|V_1|>|V_2|,$ is a bipartite MC if and only if it is a slim bipartite graph of nullity 1 with $CV=V_1$.
\end{theorem}

\begin{proof}
Let $G(V_1,V_2,E)$ be a  bipartite MC, $|V_1|>|V_2|$. Then it has nullity 1 and $|V_2|=|V_1|-1$. The set $V_1$  is $CV$ and  
$V_2$  is $CFV={\mathcal P}.$  Thus it has no $CFV_R$  and is therefore a slim graph of nullity 1.

Conversely, let $G(V_1,V_2,E)$ be a slim bipartite graph of nullity 1, with $CV=V_1$.  Then $V_2=CFV $  and by Theorem \ref{TheoBipNull1} (ii), $|V_2|=|V_1|-1.$  
Removal of $V_2$ leaves the core $F,$  induced by $CV,$  with nullity $|CV| $  increasing the nullity from 1 to $|V_1|.$ But then the nullity increases by one with the  removal  of each vertex in $V_2.$ 
 Thus $\mathcal{P}=V_2$ and  is an independent set. Also $\eta (F)= |V_1|.$
 Moreover $|\mathcal{P}|=|V_2|=\eta(F)-1.$  
 Hence  $G$ is a  bipartite MC.
  \end{proof}

It is worth mentioning that stipulating that a MC is bipartite can do away with the third axiom of a general MC.

\section{Nullspace Vertex Partition in Trees}  \label{SecTrees}
Trees are the most commonly studied class of graphs \cite{SciFioGutTreesMaxNull05}. In this section we explore MC trees and singular trees in general. First we need a result on the number of core vertices adjacent to any vertex of a singular graph on more than 1 vertex.
\begin{lemma} \label{Lem2CV}
	A vertex of a singular graph cannot be adjacent to exactly one  core vertex.
\end{lemma}

\begin{proof}
A graph is singular if there exists ${\bf x}\in {\mathbb R}^n,\ {\bf x}\not = {\bf 0},$  such that ${\bf Ax}={\bf 0}.$ 
 Let $v\in V(G).$  The $v$th row of ${\bf Ax}={\bf 0}$  can be written as $\sum_{i\sim v } x_i=0.$ 
  The neighbours of $v$ may be all $cfv$. If not, then  there exists  $w\in CV$  such that  $w\sim v$  and $x_w\not =0.$ But then there exists at least one other $cv\  w',\  w'\sim v$  with $x_{w'}\not =0$ to satisfy $\sum_{i\sim v } x_i=0.$ 
\end{proof}

As a result of Lemma \ref{Lem2CV}, if 2 core vertices are adjacent then an infinite path is a subgraph of a finite tree, since a tree has no cycles. This contradiction proves  the following  result
\begin{proposition}  \cite{NeumaierEssential82, JAUME2018836} \label{PropTreeCVindep}
	Let $T$ be a singular tree. Then $T$ has independent core vertices.
\end{proposition}

%
%

For a tree, the combinatorial properties  of the subgraph induced by $CFV_R$ will prove useful in Theorem \ref{TheoSubdiv}.

\begin{theorem} \label{TheoPM}
For a core--labelling of a singular tree $T,$   the   subgraph induced by $CFV_R$  has a perfect matching. 
\end{theorem}

\begin{proof}
In  Proposition \ref{TheoCFVrInv}, we show that $M$ as in (\ref{EqCoreL})  is invertible. The nullity $\eta(T)=n-2t=0.$  Hence the subgraph induced by $CFV_R$ has a perfect matching (a one--factor). 
\end{proof}

We shall now use the concept of subdivision for the proof of the characterization of a MC tree.

\begin{definition}
A {\it subdivision} $S$  of a connected graph $G$ on $n$ vertices and $m$ edges is obtained from $G$ by inserting a vertex of degree 2 in each edge. Thus $S$ has $n+m$ vertices and  $2m$  edges. 	
\end{definition}

\begin{lemma} \label{LemSubdv}  Let ${\bf B}$ be the vertex--edge incidence matrix of   a connected graph $G.$
The characteristic polynomial of the subdivision $S$ of a connected graph $G$ is 
$\phi(S,\lambda)= \lambda ^{n-m}\det(\lambda ^2{\bf I}-{\bf B}^{\intercal}{\bf B}).$
\end{lemma}

\begin{proof}
The adjacency matrix of $S$  is 
\begin{equation} \label{MatrixSubDiv}
 {\bf A}(S) =\left(
\begin{array}{c|c}
	{\bf 0}_n&{\bf B}\\ \hline
	{\bf B}^{\intercal}& {\bf 0}_m
\end{array}
\right).
\end{equation}

Expanding using  Schur's complement, $\phi(S,\lambda)= \lambda ^{n}\det(\lambda {\bf I}-{\bf B}^{\intercal}( \lambda {\bf I})^{-1}{\bf B} )
 =  \lambda ^{n-m}\det(\lambda ^2{\bf I}-{\bf B}^{\intercal}{\bf B}).$
\end{proof}
 \begin{corollary}  \label{CorTreeBrank}
 For a tree $T,$  the incidence matrix ${\bf B}$  has full rank.
   \end{corollary}  
\begin{proof}
 Consider ${\bf B}^{\intercal} \left(\begin{array}{c} \alpha _1\\
 \alpha _2\\
 \vdots\\
 \alpha _n \end{array}\right)= \left(\begin{array}{c} 0\\
 0\\
 \vdots\\
 0 \end{array}\right).$  Since there are only 2 non--zero entries in each column of ${\bf B},$ 
\ $\alpha _u=-\alpha _w $
for edge $\{u,w\}.$ For a   connected graph, it follows that the nullspace of ${\bf B}^{\intercal} $ has dimension 1 for a bipartite graph and 0 otherwise.   The tree $T$ is bipartite and $m=n-1.$ Hence the rank of ${\bf B} $ which is the same as the rank of ${\bf B}^{\intercal} $ 
  is $m.$ 
\end{proof}

\begin{corollary}  \label{CorSub}
The subdivision of a tree is singular with nullity 1.
\end{corollary}

\begin{proof}
This follows immediately from Lemma \ref{LemSubdv}  since from Corollary \ref{CorTreeBrank},  the nullspace  of ${\bf B}^{\intercal}{\bf B}$ is  $\{\bf 0\}$ for a tree with  $m=n-1.$
\end{proof}

In \cite{LAMAMinConfigTreesGutmanSciriha_2006}, a characterization of MC trees is presented.  Here we give a different proof by using  Corollary \ref{CorSub}.

\begin{theorem}  \label{TheoSubdiv} \cite{LAMAMinConfigTreesGutmanSciriha_2006}
	A tree  is a minimal configuration if and only if it is a subdivision of another tree.
\end{theorem}
\begin{proof}
Let $T'$  be a MC with $|CV|=n$   and $|{\mathcal P}|=|N(CV)|= m.$  Then $m-n=1. $ Note that both $CV$  and $N(CV)$  are independent sets, the partite sets of $T'.$  Also the number of edges of $T'$  is $m+n-1=2m.$ Now a vertex of ${\mathcal P}(T')$ cannot be an end vertex as otherwise its neighbour is a $cfv$, contradicting the independence of ${\mathcal P}(T')$ in a MC. Thus each vertex of ${\mathcal P}(T')$  has degree 2. Therefore $T'$ is the subdivision of a tree $T$ on $n$ vertices and $m$  edges.

Conversely, 
 let   $T$ be a tree on  $n$ vertices and $m$ edges and let   $S$ be its subdivision. 
  Then by  Corollary \ref{CorSub},  $S$  has nullity 1.
%

   Since $S$  is a singular tree, then by Proposition \ref{PropTreeCVindep}, CV is an independent set. Hence $S$  has a core--labelling. 
Let the partite sets $V_1$  and $V_2$ 
in $S$ be the original vertices of $T$  and the inserted vertices, respectively. Note $|V_1|=|V_2|+1.$ 
By Theorem \ref{TheoBipNull1} (iii), $CV\subseteq V_1.$  Since $S$ is bipartite, $N(CV)\subseteq V_2.$

Recall that $V_1$ in $S$ was the set of original vertices of $T.$ 
Let $w\in V_1.$ 
The subgraph $S-w$ of $S,$ obtained from $S$ after  removing $w$ has a perfect matching with edges $\{u_i,w_j\}, \ u_i \in V_2 ,\ w_j\in V_1.$  Hence $S-w$ has nullity 0. This means that the nullity of $S$ decreases on deleting $w.$ Hence $w\in CV,$  that is $V_1\subseteq CV.$  The subset $V_1$ is therefore $CV$ in $S.$

We now consider $V_2,$ which is a partition of $N(CV)$  and $CFV_R.$ 
Since the $S$ is connected, then $V_2=N(CV). $
 It follows that $S$ is a bipartite slim graph of nullity 1, with $V_1=CV$. By Theorem \ref{TheoBipMCharzn}, $S$ is a bipartite $MC$.

%
%
%
%
%
%
%
%
%
%
%
%
%
%
%
%
\end{proof}


Note that the subgraph of $S,$ obtained after  removing $u\in V_2,$  is a subdivision of a forest of two trees and has nullity 2. Repeating the process until all the  vertices in $V_2$ are removed, the nullity increases to $V_1.$  Hence the nullity increased by 1 with each vertex  deletion. It follows that each vertex in $V_2$ is an upper $cfv,$  a condition required for a MC.
It is also worth noting that the incidence matrix $\bf B$ appearing as a submatrix of the adjacency matrix of a subdivision  of a tree in  (\ref{MatrixSubDiv})  is precisely ${\bf Q}$  in (\ref{EqCoreL}). 


 We now show that the size of the periphery of a MC tree is related to the matching number $t.$

For a general singular tree  $T$, 
a maximal matching consists of 
the  pendant edges removed, until $\overline{K_{\eta(T)}}$ is obtained,  
starting from any end--vertex in $T.$ 
One can start from  a slim forest and  extend to a general tree $T'$ of the same nullity with the CV preserved by adding pairs of adjacent vertices in  $CFV_R(T').$ This can be done either by adding a pendant edge and 
  joining it to a $cfv $ or by inserting two vertices of degree 2 in an edge with $cfv$s 
    as end vertices. 

\begin{proposition} \label{ProptMC}
	If $T'$ is a minimal configuration  tree, then $t = |N(CV)|$.
\end{proposition}
\begin{proof}
	For a MC tree, $\eta(T')=1=n(T')-2t.$ Also, by Theorem \ref{TheoSubdiv}, $T'$ is the subdivision of a tree $T$ on $n$ vertices and $m$ edges. So $n(T')=n+m $  and $2t= n+m-1= 2m.$  Since the vertices in $N(CV(T'))\ (= {\mathcal P}(T'))$  are the vertices inserted 
	in the edges of $T$ to form the subdivision, $t=m=|N(CV)|$.
\end{proof}


%


 The next  result is on the rank of ${\bf Q}$ in the adjacency matrix of a 
core--labelled tree.

\begin{theorem} \label{TheoTreeQ}
	If $T$ is a core--labelled tree, then the columns of $\mathbf{Q}$ are linearly independent.
\end{theorem}

\begin{proof} For a  core--labelled graph $G$,
by Theorems  \ref{TheoNullCVrnkQ},    $\eta(G) = |CV| - \textup{rank}(\mathbf{Q})$. 
For a tree, $\eta(T)=n-2t.$  By Theorem \ref{TheoPM}, 
for a core--labelled tree,   $2t=2|N(CV)|+|CFV_R|. $ 
Since $n=  |CV|+ |N(CV)|+|CFV_R|$, by eliminating t,    $\eta(T) = |CV| - |N(CV)|$. 
By Theorem  \ref{TheoQindep}, $\mathbf{Q}$  has full rank and the  $|N(CV)|$  columns of $\mathbf{Q}$
are linearly independent.
  
\end{proof}

\section{Nullspace Preserving Edge Additions}  \label{SecAddE}
In this last section, we explore which edges could be added (or removed) from a graph to preserve the nullity or the core vertex set.

By Cauchy's Interlacing Theorem for real symmetric matrices, the nullity changes by at most 1,  on adding or deleting a vertex. By definition, if the vertex is a $cfv_{mid},$  the nullity is preserved. We now explore which edge  additions allow  the nullity and the core vertex set  to be  both preserved  in  a graph with independent core--vertices. We use  again the vertex partition into $CV$, $N(CV)$  and $CFV_R$  induced by a core--labelling.   We consider adding an edge between two  vertices within a part or between two distinct parts of  the partition.

\begin{theorem} \label{TheoEcvNcv}
	Let $G$ be a core-labelled graph. 
	 Let $u \in CV$ and $w \in N(CV)$, such that $u \nsim w$ in $G$. Let $G^\prime \coloneqq G + e$ be obtained from $G$ by adding an  edge $e$ such that the core-labelling is preserved, where $e \coloneqq \{u,w\}$. Then $\eta(G^\prime) \geq \eta(G).$ Moreover, there is a vector $\mathbf{x}_{ _{CV}}$ which is in $\text{\rm Ker}\left(\mathbf{Q}^\intercal\right)$ but not in $\text{\rm Ker}\left(\left(\mathbf{Q^\prime}\right)^\intercal\right)$ and a vector $\mathbf{y}_{ _{CV}}$ which is in $\text{\rm Ker}\left(\left(\mathbf{Q^\prime}\right)^\intercal\right)$
	 but not 
	in $\text{\rm Ker}\left(\mathbf{Q}^\intercal\right)$ .
\end{theorem}

\begin{proof}  For a core labelling of a graph $G,$  with  vertices $u \in CV$ and $w \in N(CV)$ labelled 1 and $|CV|+1,$  respectively,  such that  $u \nsim w$ in $G, $  we write  $u = 1$ and $w = |CV| + 1.$
 Let the adjacency matrix  $\mathbf{A}$ be  as in (\ref{EqCoreL}). 
On adding edge  
   $\{u,w\},$
 the adjacency matrix $\mathbf{A}^\prime$  of $G^\prime$ 
   satisfies  $\mathbf{A}^\prime = \mathbf{A} + \mathbf{E}$ where  \vspace*{-.52mm}
	$$\mathbf{E} = \left(\begin{array}{c|c|c}
	0 & 1 \ 0 \ . . . & 0 \\  \vspace*{-.35cm}
	&&\\ 
	{\bf 0} & {\bf 0}  \  & {\bf 0} \\ \hline
	1 \ 0 \ . . . &\ 0 \ . . . & 0 \\   \vspace*{-.35cm}
	&&\\  
	{\bf 0} \ & {\bf 0} &{\bf 0} \\ 
			\hline
	{\bf 0} &{\bf 0} & {\bf 0}\\
	\end{array}\right).$$   

Since $u$ is a cv, there exists  $\mathbf{x}^{(1)}$ in the nullspace of  $ \mathbf{A}$   with the first entry $\alpha $   non-zero.
If $\eta(G)>1,$	let $\mathbf{x}^{(1)}, \mathbf{x}^{(2)}, ..., \mathbf{x}^{\left(\eta(G)\right)}$ 
 be a basis for the nullspace of $G$, such that only $\mathbf{x}^{(1)}$ has the first entry non-zero.
	Denoting column $i$ of the identity matrix  by   $\mathbf{e}_i$  
	  and writing $\mathbf{x}^{(1)}= \left(\begin{array}{c}
	  	 \mathbf{x}_{ _{CV}}\\ \hline {\bf 0} \\ \hline {\bf 0} \end{array}\right),$ conformal \mbox{with (\ref{EqCoreL}),}  row  $w$  of  $\mathbf{A}^\prime \mathbf{x}^{(1)}$ is $\mathbf{e}_w^{\intercal}(\mathbf{Q}^\prime)^{\intercal} \mathbf{x}_{ _{CV}} =   
   \mathbf{e}_w^{\intercal}\mathbf{Q}^{\intercal} \mathbf{x}_{ _{CV}} +\alpha =\alpha \neq { 0}.$ 
   
    Hence  $\mathbf{Q}^\prime \mathbf{x}^{(1)}_{ _{CV}}\neq  {\bf 0}.$ 
   By the proof of Lemma \ref{LemNullQ},  $\mathbf{A}^\prime \mathbf{x}^{(1)}\neq {\bf 0}.$ 
   Thus $\mathbf{x}^{(1)}$  is a vector in the nullspace of $\mathbf{A}$ but not in the nullspace of 
   $\mathbf{A}^\prime.$ Moreover, $\left(\mathbf{Q^\prime}\right)^\intercal \mathbf{x}^{(i)}_{ _{CV}} = \mathbf{0}$, for $2 \leq i \leq \eta(G)$. 
   Thus 
    the $\eta(G) - 1$ vectors $\mathbf{x}^{(2)}, ..., \mathbf{x}^{\left(\eta(G)\right)}$ lie in the nullspace of $G^\prime$.
    
     Since $CV $  is preserved of adding edge $\{u,w\}$,  $u$ is also a core vertex in $G^\prime$. 
     Hence there is another vector $\mathbf{y}^{(1)}$  in the nullspace in  $ \mathbf{A}^\prime$ with the first entry non-zero.    Therefore $\eta(G^\prime) \geq \eta(G)$.\\
     
   A similar  argument as above yields  $\eta(G^\prime) \leq \eta(G)$, 
   so that the graphs $G$ and $G^\prime$  have the same nullity. 
    Moreover, $\mathbf{x}^{(1)} $ is a vector in the nullspace of $\mathbf{A}$ but not in the nullspace of $\mathbf{A}^\prime$  whereas $\mathbf{y}^{(1)}$  is a vector in the nullspace of $\mathbf{A}^\prime$ but not in the nullspace of $\mathbf{A}.$   \end{proof}

As a consequence of Theorem \ref{TheoEcvNcv}, addition of an  edge from a vertex in $CV$ to a vertex in  $ N(CV)$ which
 preserves  the  core-labelling does not change the nullity but may change  the nullspace. The  addition of edges between two vertices in $CV$ vertices is not possible as  the core-labelling will  not remain well defined. Furthermore, the addition of an edge between a $CV$ vertex and a $CFV_R$ vertex is not permissible either as the core-labelling changes.

Therefore,  to preserve the core--labelling,   only the following edge additions are left to be considered: \vspace{-2mm}
 \begin{enumerate}[(i)]
	\item $N(CV)$ -- $N(CV)$ edges,
	\item $N(CV)$ -- $CFV_R$ edges, 
	\item $CFV_R$ -- $CFV_R$ edges.
\end{enumerate} 
\vspace*{-.5cm}
Before presenting results on the perturbations that satisfy constraints relating to the nullspace of ${\bf A},$ we give examples to show  the possible effects on the vertex types and on the nullspace  on adding an edge to graphs with independent core vertices.

\begin{figure}[!h]
\begin{center}
\includegraphics[trim= 0cm 0cm 0cm 5cm,width=6cm]{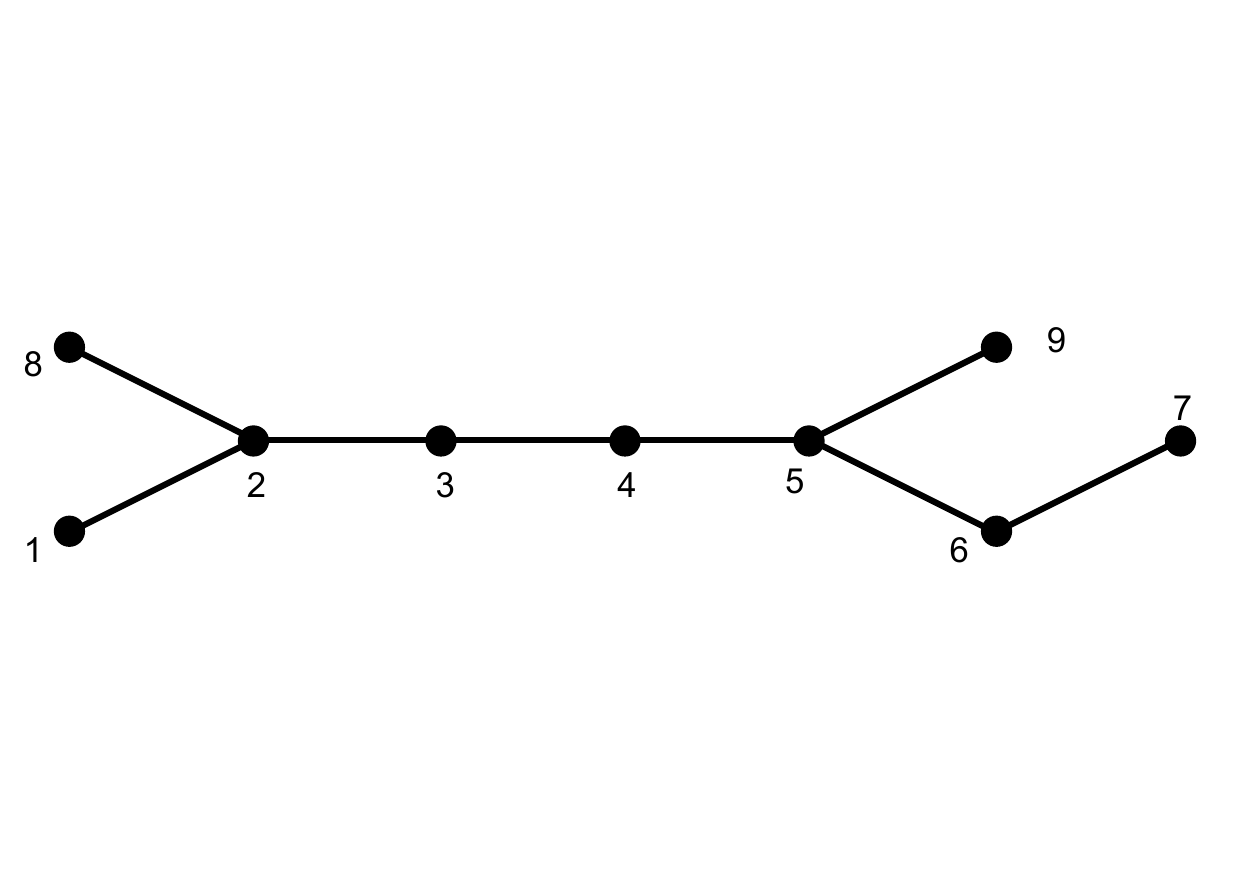} \hspace*{.1cm}
\includegraphics[trim= 0cm 0cm 0cm 1cm, width=5cm]{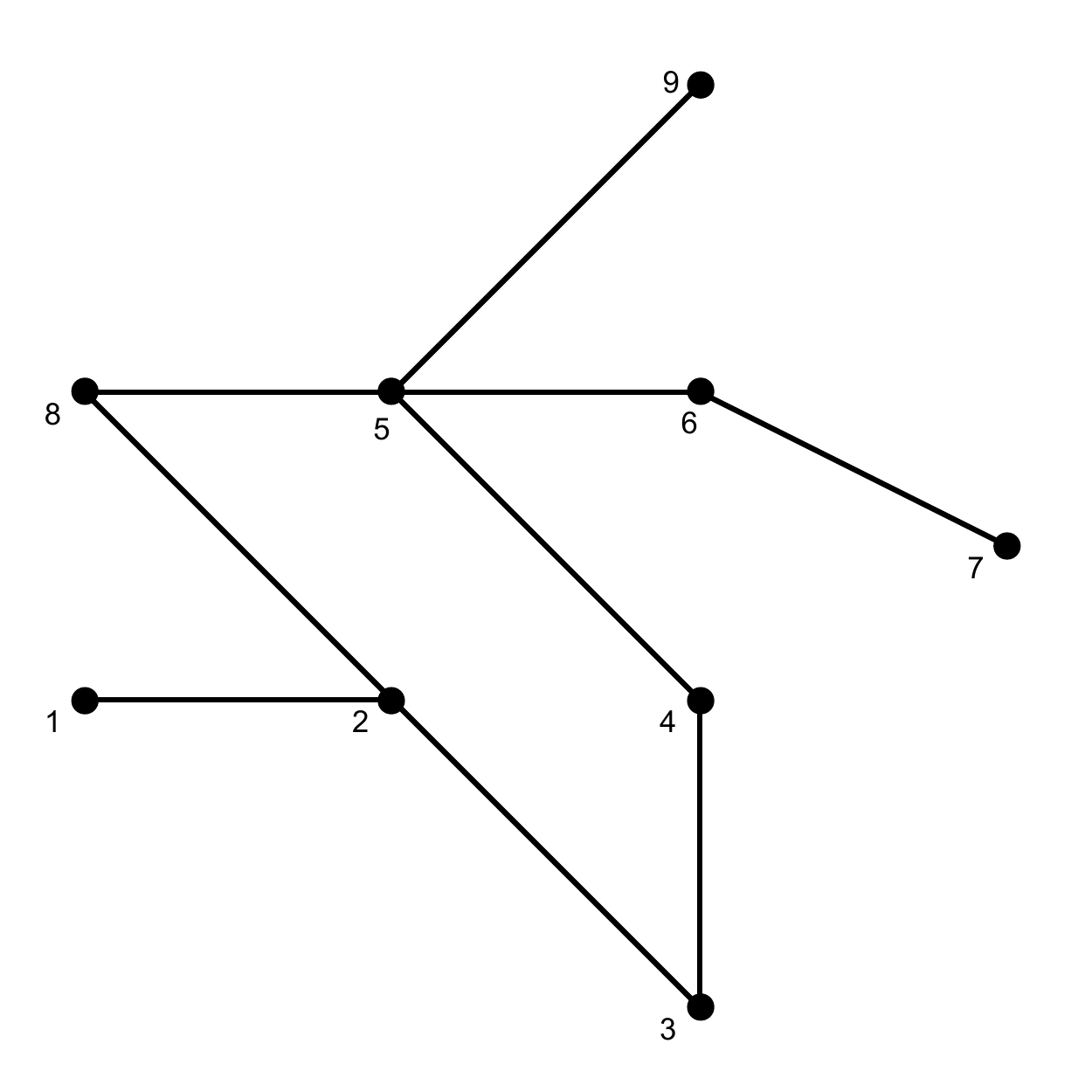}  \vspace*{-.3cm}
\caption{Adding edge $e=\{5,8\}$ to the tree $T$ of nullity 1 preserving the nullity but altering the   core--vertex set.}  \label{FigEdge85}
\end{center}
\end{figure} \vspace{-4mm}
Figure \ref{FigEdge85} shows tree $T$  and the unicyclic graph $T+e_{\{5,8\}} $  with core vertices $\{1,8\}$  replaced by $\{1,8,9\}.$  Figure \ref{FigEdge314}  shows the half cores $H$ of nullity 2 and $H+e_{3,14}$ with the same core vertices but with different nullspace vectors of their adjacency matrix.  The nullspace of ${\bf A}(H)$  is generated by   \vspace{-4mm}
$$\{\{0, -1, 0, 0, 0, 1, 0, -1, 0, 1, 0, -1, 0, 1\}, \{0, -1, 0, -1, 0, 1, 
  2, 0, 0, 0, 0, 0, 0, 0\}\}$$  
  
   \vspace{-4mm}
 and on adding the edge $\{3,14\}$  the nullspace  generator  of ${\bf A}(H+e_{\{3,14\}})$ becomes  
 \vspace{-3mm}
$$\{\{0,-1,0,1,0,1,0,-2,0,2,0,-2,0,2\},\{0,-1,0,-1,0,1,2,0,0,0,
0,0,0,0\}\}.$$
\vspace*{-.32cm}
  
\begin{figure}[!h]
\begin{center}
\includegraphics[trim= 0cm 2cm 0cm 3cm, width=7cm]{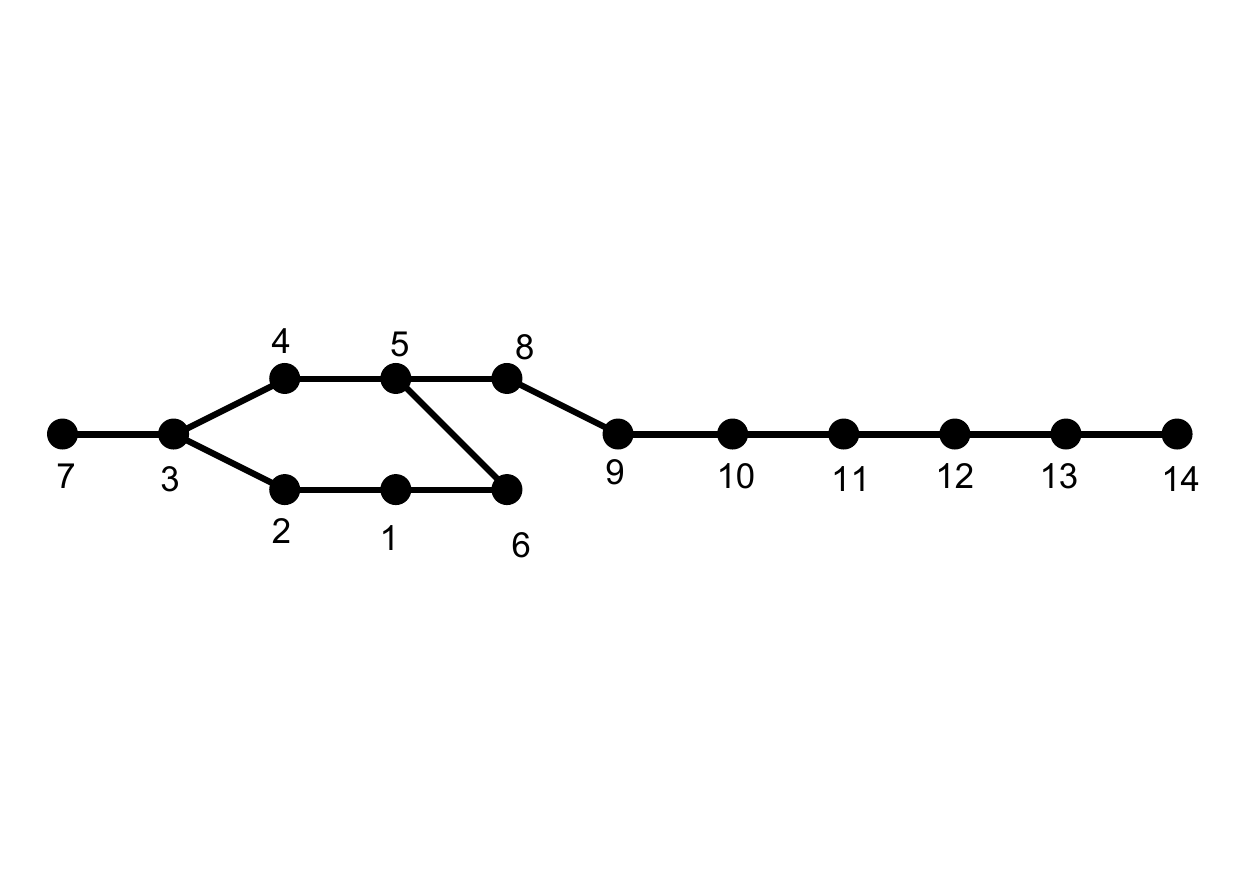} 
\includegraphics[trim= 0cm 0cm 0cm 3cm, width=6cm]{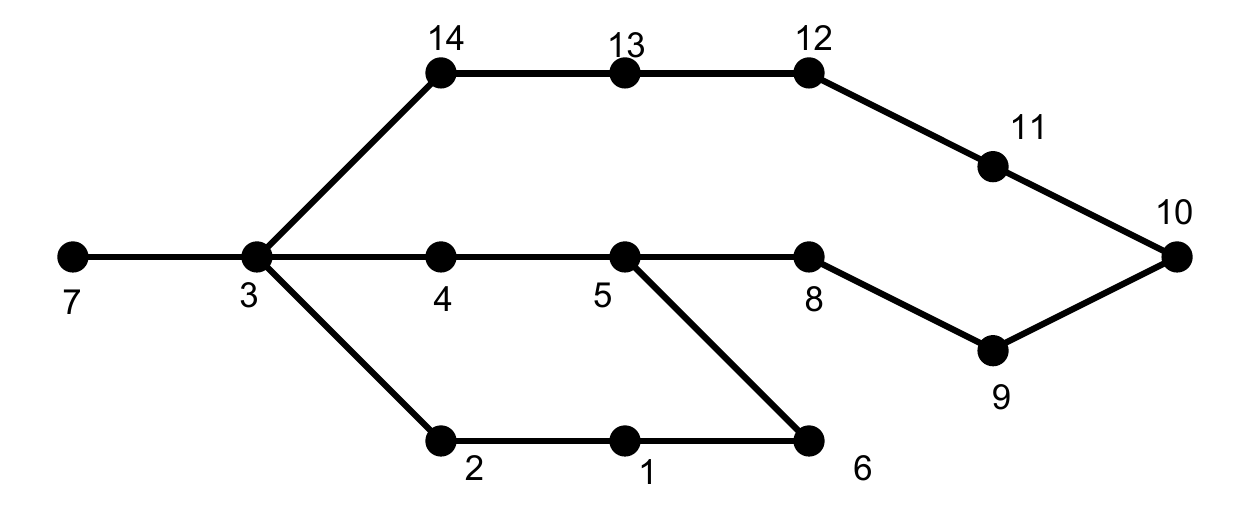}
\caption{Adding edge $e=\{3,14\}$ to the graph  $H$ of nullity 2 preserving the nullity  and the   core--vertex set  but altering the   nullspace.}  \label{FigEdge314}
\end{center}
\end{figure}
\pagebreak

We give another example where the nullity changes from 0 to 2 on adding an edge. The perturbation to the tree $T'$ shown in Figure \ref{FigEdge12} is the addition of edge $\{1,2\}$.  The nullspace of ${\bf A}(T')$  is generated by 
$\{0\}$
and on adding the edge $\{1,2\}$  the nullspace  generator  of ${\bf A}(T'+e_{1,2})$ becomes  
$\{\{0, 1, 0, -1, 0, 1\}, \{-1, 0, 1, 0, 0, 0\}\}.$

\begin{figure}[!h]
\begin{center}
\includegraphics[trim= 0cm 2cm 0cm 2cm, width=7cm]{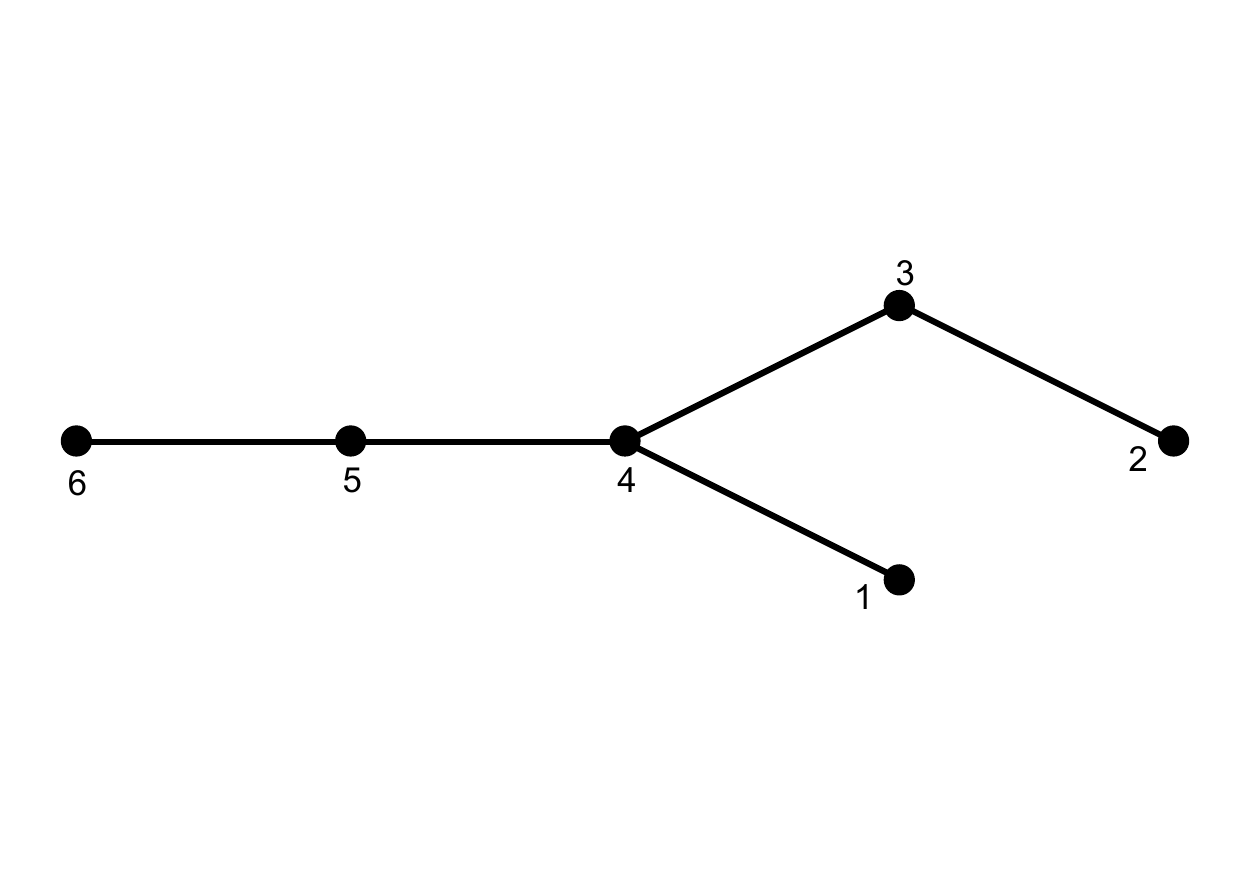} 
\includegraphics[trim= 0cm 2cm 0cm 2cm, width=5cm]{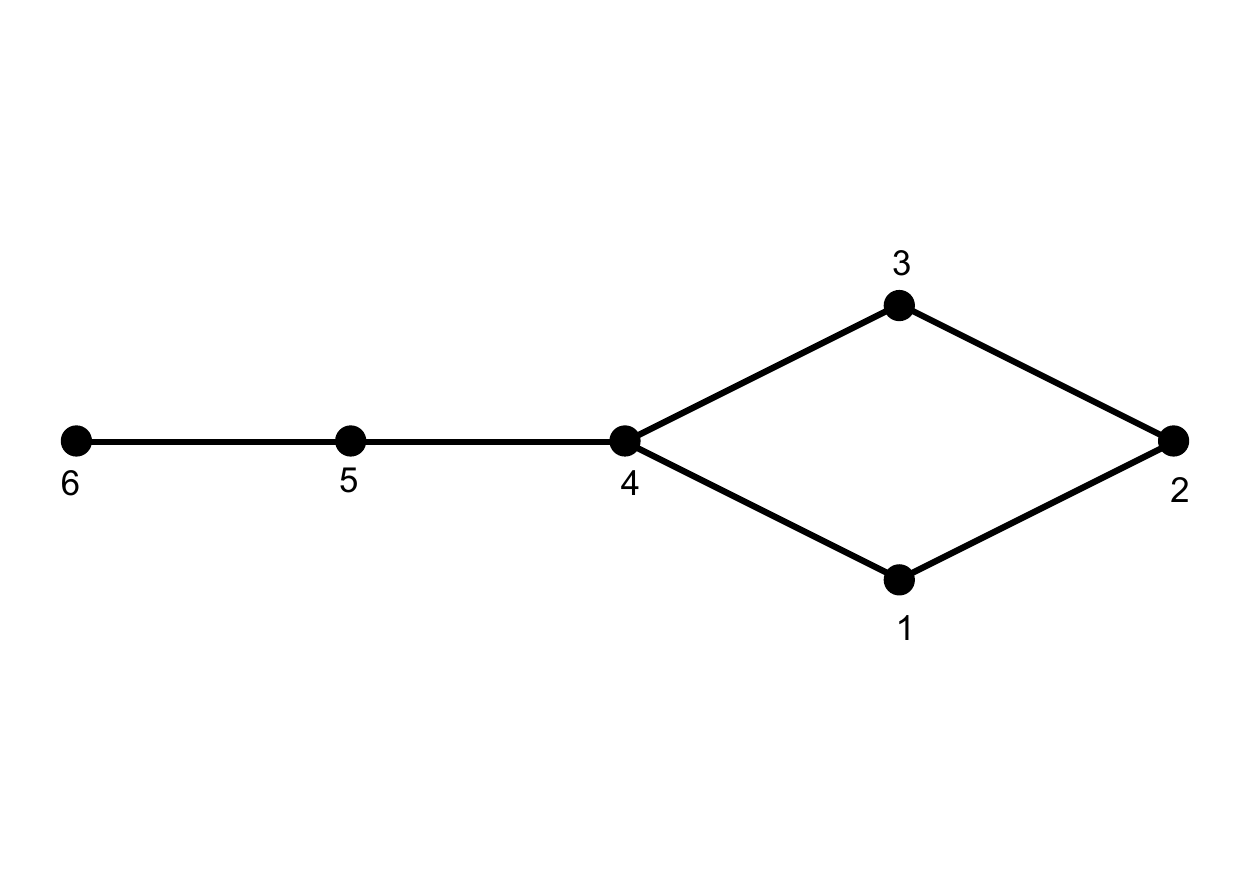}
\caption{Adding edge $e=\{1,2\}$ to the non--singular tree  $T'$ increases the nullity to 2   and creates a four    core--vertex set.}  \label{FigEdge12}
\end{center}
\end{figure}

\begin{proposition} \label{PropG+e}
	Let $G$ be a  graph with independent core vertices. Let $u$ and $w$ be core-forbidden vertices, such that $u \nsim w$ in $G$. Let $G^\prime \coloneqq G + e$ be obtained from $G$ by adding the edge  $e = \{u, w\}$. If the  nullity is preserved, then $G + e$ has the same nullspace and core-labelling of $G$.
\end{proposition}

\begin{proof}
	Let $G$ be labelled so that  $\mathbf{A}$ is a  block matrix as in (\ref{EqCoreL}).  We show that a kernel vector $\mathbf{x}$ of $\mathbf{A}(G)$ is a kernel vector for $\mathbf{A}(G^\prime)$.
	
	Let $\mathbf{x}_{ _{CV}}$ be the restriction 	
	$\left(\alpha_1,\ldots, \alpha_{|CV|}\right)^\intercal $	
	of $\mathbf{x}$ to the core vertices of $G$. Then $\mathbf{x} = (\mathbf{x}_{ _{CV}},{\bf 0}).$ By definition of a kernel vector, $\mathbf{A}(G)\mathbf{x} = \mathbf{0}$. Therefore $\mathbf{Q}^\intercal  \mathbf{x}_{ _{CV}} = \mathbf{0}.$
	
	Now, on adding edge $e,$ the change in $\mathbf{A}(G)$ is contained in the blocks associated with  the core-forbidden vertices. Therefore, $\mathbf{A}(G^\prime)\mathbf{x} = 
	\mathbf{Q}^\intercal  \mathbf{x}_{CV} = \mathbf{0}.$
	
	Therefore the kernel vectors of $\mathbf{A}(G)$ are also kernel vectors of $\mathbf{A}\left(G^\prime\right)$. Thus $CV(G) \subseteq CV\left(G^\prime\right)$, that is $\eta(G) \leq \eta\left(G^\prime\right)$. If a \textit{cfv} in $G$ becomes a \textit{cv} in $G^\prime$, then the nullity increases. But the nullity is preserved. Hence $CV$ is preserved and so is the nullspace. In turn, it follows that $N(CV)$ and core-labelling of $G$ are unaltered by the perturbation. 
\end{proof}

The necessary condition established in Proposition \ref{PropG+e} can be relaxed to 
   a necessary and sufficient condition involving $CV$ only.

\begin{theorem} \label{TheoAddE}
	Let $G$ be a graph with independent core vertices. Let $u$ and $w$ be core-forbidden vertices, such that $u \nsim w$ in $G$. Let $G + e$ be an edge addition to $G$, where $e = \{u, w\}$. Then, nullity is preserved if and only if $CV(G) = CV(G + e)$.
\end{theorem}

\begin{proof}
	Let the nullity be preserved. By Proposition \ref{PropG+e}, it follows that the core-labelling is preserved and hence $CV(G) = CV(G + e)$.\\
	
	Conversely, let $CV(G) = CV(G + e)$. Since the added edge is amongst the core-forbidden vertices in $G$, then $\mathbf{Q}(G) = \mathbf{Q}(G + e)$. By Theorem \ref{TheoNullCVrnkQ}, \vspace{-.35cm}
	\begin{align*}
		\eta(G) &= |CV(G)| - \text{rank}(\mathbf{Q}(G)) \\
		&= |CV(G + e)| - \text{rank}(\mathbf{Q}(G + e)) \\
		&= \eta(G + e)
	\end{align*} \vspace*{-.5cm}
	and hence nullity is preserved.
\end{proof}

The study of perturbations to networks finds many applications, in information technology and social networks in particular \cite{WANG20122084,RowlinsonPerturb1990,SciBriffDups2019}.
The results presented here are  of interest in combinatorial optimization and the study of perturbations to singular networks with the goal of  inserting   or removing  edges efficiently  while maintaining the same core vertex set. 
In machine learning,  to train a neural network,  switches linked to edge detectors in  the neural network  stochastically  disable specific  detectors in accordance with a preconfigured probability.  This technique is used to reduce over--fitting on the training data
\cite{GoogleDropOut2019}.
The behaviour of graph invariants, when applying changes to a graph   with constraints associated with the nullspace  of the adjacency matrix, leads to optimal architectures with a specified  nullity, retaining the independence of the core vertex set or  the core--labelling.

Many algorithms in predictive modelling depend on the processing of network graphs with underlying spanning trees in a network.
 The combinatorial properties of trees that we discussed  shed light on their inherent structure and help to devise efficient algorithms.
In the search for optimal network graphs with a constraint related to the nullspace of the adjacency matrix, one may start with a slim graph and add an admissible edge  joining non--adjacent vertices. The goal can be the preservation of one or more of the three properties associated with the nullspace of the adjacency matrix. These are the nullity, the core--vertex set  and the  entries of the normalized basis vectors of the nullspace of the adjacency matrix.

 Depending on the property to be preserved, edges can be added  selectively  to obtain optimal networks with a maximal number of edges  having the constant property. 
We have shown that adding edges to a graph may alter the  core vertex set,    the nullity  or the nullspace. Constraints may be imposed to keep one aspect unchanged.   Theorem  \ref{TheoEcvNcv} shows that adding edges between the mixed types $CV$ and $N(CV)$ of vertices, while the core--labelling is unchanged, preserves the nullity but upsets the nullspace.   By Theorem  \ref{TheoAddE}, adding edges between core--forbidden vertices is a  safe operation since  the core vertex set is left intact,  as long as the nullity  is unaltered.

\vspace{-.5cm}

%
%
%
%
%
%
%
%
%

\end{document}